\documentclass[12pt,reqno,a4paper]{article}

\usepackage[latin1]{inputenc}

\usepackage{amsmath}
\usepackage{amsthm}
\usepackage{mathenv}
\usepackage{enumerate}
\usepackage{amssymb}

\def\A{\mathfrak{A}}

\def\alg{\hbox{\rm alg}}

\def\:{\colon}

\newtheorem{Th}{Theorem}
\newtheorem{Le}{Lemma}

\title{ON CONMUTATIVE LEFT-NILALGEBRAS OF
INDEX 4\footnote{ Partially supported by FAPESP 05/01790-9, Brasil}}

\author{ \it JUAN C. GUTIERREZ FERNANDEZ\\ \it UNIVERSIDAD DE SÃO PAULO, BRAZIL}
\date{\it 2018}

\begin{document}
\maketitle
{\small 

\begin{abstract} \it 
We first present a solution to  a conjecture of \cite{CHL} in the positive.  
 We prove that if $\A$ is a commutative nonassociative algebra  over a field of characteristic $\ne 2,3$,
satisfying the identity  $x(x(xx))=0$, then 
 $L_{a^{t_1}}L_{a^{t_2}}\cdots L_{a^{t_s}} \equiv  0$ if $ t_1+t_2+\cdots +
 t_s \geq 10$, where $a\in \A$. \rm 
 \end{abstract}
 \vspace{1cm}

 \vspace{1cm}
  
  
  }

\newpage
 
 \section{Introduction}

Throughout this paper  the term \it algebra \rm is understood to be a 
commutative not necessarily associative algebra. We will use the notations
 and terminology of \cite{F}. 
Let $\A$ be an (commutative nonassociative) algebra over a field $F$. We define inductively 
the following powers, $\A^1 = \A$ and $\A^{s} = \sum_{i+j = s} \A^i \A^j$ 
for all positive integers $s\geq 2$. We shall say that $\mathfrak{A}$ is 
\it nilpotent \rm in case there is an integer $s$ such that $\A^s = (0)$.
The algebra  $\mathfrak{A}$ is called  \it
nilalgebra \rm in case 
the subalgebra $\alg(a)$ of $\A$ generated by $a$ is nilpotent, for all $a\in \A$. Therefore
$\A$ is nilalgebra if and only if  for every $a\in \A$ there exists an integer $t$ such
that every product of at least $t$ factors each of them equal to $a$, in
whatever association, vanishes. The \it (principal) powers \rm of an element $a$ in $\A$ are defined recursively by 
 $a^1 = a$ and $a^{i+1} = a a^i$ for all  integers $i\geq 1$. The algebra $\mathfrak{A}$ is called 
  \it right-nilalgebra \rm if for every $a$ in $\mathfrak{A}$ there exists an integer $k=k(a)$ such that 
  $a^k =0$. The smallest positive integer $k$ which this property is the \it index\rm. 
   Obviously, every nilalgebra is  right-nilalgebra.  For any element $a$ in $\mathfrak{A}$, the linear mapping  
   $L_a$ of $\mathfrak{A}$  defined  by $x\rightarrow ax$ is called \it multiplication  operator \rm of 
   $\mathfrak{A}$. An \it Engel algebra \rm is an algebra in which  every  multiplication operator is nilpotent in the sense that for every $a\in\mathfrak{A}$ there exists a positive integer $j$ such that $L_a^j = 0$.

An important  question is that of the existence of simple nilalgebras
in the class of finite-dimensional  algebras.
In \cite{F} we proved that every nilagebra $\A$ of dimension $\leq 6$ over a
field of characteristic $\ne 2,3,5$ is solvable and hence $\A^2 \varsubsetneqq
\A$. For power-associative nilalgebras of dimension $\leq 8$ over a field of
characteristic $\ne 2,3,5$, we have shown  in \cite{FS} that they are
solvable, and hence there is no simple algebra  in this subclass. See also 
\cite{ES} and \cite{F}   for power-associative nilalgebras of
dimension $\leq 7$.

We show now the  process of linearization of identities, which  
is an important tool in the theory of varieties of algebras. See 
\cite{G}, \cite{O} and  \cite{Z} for more information.
 Let $P$ be the free commutative nonassociative
 polynomial ring in two generators $x$ and $y$ over a field $F$.
For every $\alpha_1,\ldots,\alpha_r\in P$, the \it operator linearization \rm 
 $\delta[
\alpha_1,\ldots,\alpha_r]$ can be defined as follows: if $p(x,y)$ is a monomial
in $P$, then $\delta[\alpha_1,\ldots,\alpha_r]p(x,y)$ is obtained by making
 all the possible replacements 
of $r$ of the $k$ identical arguments $x$ by $\alpha_1,\ldots,\alpha_r$ and
summing the resulting terms if  $x-$degree of $p(x,y)$ is $\geq r$,
and is equal to zero in other cases. Some examples of this operator are
\begin{eqnarray*}[ll]
\delta[y](x^2(xy)) &= 2 (xy)^2+ x^2y^2\\
\delta[x^2,y](x^2) &= 2 x^2y, \quad \delta[y,xy^2,x](x^2) = 0.
\end{eqnarray*}
For simplicity, $\delta[\alpha:r]$ will
 denote $\delta[\alpha_1,\ldots,\alpha_r]$, where $\alpha_1 = \cdots =
 \alpha_r = \alpha$. We observe that if $p(x)$ is a polynomial in $P$,
  then $p(x+y) = p(x) + \sum_{j=1}^\infty \delta[y:j] p(x)$, where $\delta[y:j] p(x)$ is the sum of all the 
  terms of $p(x+y)$ which have degree $j$ with respect to $y$.

\begin{Le}[\cite{Z}]
Let $p(x,y)$ be a commutative nonassociative polynomial of  $x$-degree  $\leq n$. If $F$ is a field of
 characteristic either zero or  $\geq n$, and the $F$-algebra $\A$
  satisfies the identity $p(x,y)$, then $A$ satisfies all linearizations of $p(x,y)$. 
\end{Le}

\section{Right-nilalgebras of index 4} 
Throughout this section  $F$ is a field of characteristic
 different from 2 or 3  and all the algebras are over $F$.
 We will study  right-nilalgebras of index $\leq 4$, that is 
 the variety $\mathcal{V}$ of  algebras over the
field $F$ satisfying  the identity
\begin{equation}\label{ident}
 x^4=0.
\end{equation}
Let $\A$ be an algebra in $\mathcal{V}$. For simplicity, we will denote by
$L$ and $U$ the multiplication operators, $L_x$ and $L_{x^2}$ respectively,
where $x$ is an element in $\A$. The following known result  is a basic tool in our investigation. See \cite{CHL} and \cite{EL}.
\begin{Le} \label{Le3} Let $\A$ be a commutative right-nilalgebra of index 4. Then $\A$ satisfies the identities
\begin{equation}\label{eqx}
x²x³ = - x(x²x²), \quad x³x³ = (x²)^3=x(x(x²x²)),
\end{equation}
and $p(x) = 0$, for every monomial $p(x)$ with $x$-degree $\geq 7$. 
Furthermore,  we have 
\begin{eqnarray}[lcl]
L_{x³} &=& -LU -2 L^3,\label{eql3}\\
L_{x²x^2} &=& -U² - 2 UL^2 -2LUL + 4 L^4,\\
L_{x(x²x²)} &=& - LU² - 2 LUL²-2L²UL - 4 L^3U - 12 L^5,\label{eql5}\\
L_{x(x(x²x²))} &=& 2 L²U² + 4 L²UL² + 4 L^4U + 8 L^6,\label{eql6}
\end{eqnarray} 
and also 

\begin{center}{\footnotesize
{\rm Table i: Multiplication identities of degree 5}\\
\begin{array}{l|ccccccc}\hline
&ULU&LU^2&UL^3&LUL^2&L^2UL&L^3U&L^5\\
\hline
U^2L& 0 & -1 & 2& 0 & 0 & -2 & -8\label{u2l}\\
\hline
\end{array}
}
\end{center}
\noindent and two identities of $x$-degree 6  which may be
written as  

\begin{center}
{\footnotesize
{\rm Table ii: Multiplication identities of degree 6}
\begin{array}{l|ccccccccc}\hline
&  UL^2U  &  (LU)^2  &  L^2U^2 & UL^4 & LUL^3 & L^2UL^2 & L^3UL & L^4U & L^6 \\
\hline
 U^3 &-2&-2&2&-8&-8&0&-4&8&40\label{u3}\\ 
 (UL)^2 &-1&-1&1&-4&-2&2&0&4&24\label{ulul}\\ \hline
\end{array}
}
\end{center}
\end{Le} 
We note that, for example,  Table i means that $U^2L = - LU^2 + 2 UL^3 - 2 L^3U
-8 L^5$. From the identities (\ref{eql3}-\ref{eql6}) we get that for any $a\in \A$ 
the associative algebra $\A_a$ generated by all  $L_c$ with $c\in \alg(a)$ is in
 fact generated by $L_a$ and $L_{a^2}$.   Furthermore,
 every algebra in $\mathcal{V}$ is  a nilalgebra of index $\leq 7$.

We now pass to  study  homogeneous identities in $\A$ with  $x$-degree $\geq 7$  and
$y$-degree 1. From the relation $0= \delta[y,x,x³,x³](x^4) = 2y(x(x^3x^3))+4y(x^3(xx^3))
+2x(y(x^3x^3))+4x^3(y(xx^3))+ 4x(x^3(yx^3))+ 4 x^3(x(yx^3))+ 4x^3(x^3(xy))
=2x((x^3)^2y)+4x(x^3(x^3y))+ 4 x^3(x(x^3y))+ 4x^3(x^3(xy))
=2[LL_{x^3x^3}+ 2LL_{x^3}L_{x^3}+ 2L_{x^3}LL_{x^3}+2L_{x^3}L_{x^3}L](y)$
 we have 
\begin{equation}
L^3U^2 = - 2 L^3UL² - L^4UL - 5 L^5U - 20 L^7,
\end{equation}
since we can use the reductions (\ref{eql3}-\ref{eql6}) and replace the
occurrences of  $(UL)^2$. Multiplying  the identity of Table i  by
$U$ from the left,  replacing first the occurrences of $U^3$ and
next using reductions from Table i, Table ii and above identity we get a new identity as
follows:
{\footnotesize\begin{eqnarray*}
0 &=&U^3L + U LU^2- 2U^2L^3 + 2UL^3U +8U L^5 = [-2U L^2UL - 2(LU )^2L +2L^2U^2L - 8U L^5\\
& &-8LU L^4- 4L^3U L^2 + 8L^4UL +40L^7] + U LU^2 - 2U^2L^3 + 2UL^3U +8U L^5\\
&=&-2UL^2UL + [2LU L^2U +2L(LU)^2 + [4L^3U L^2 +2L^4UL + 10L^5U +40L^7]\\
& &+8LU L^4 + 4L^2U L^3- 4L^3UL^2 -8L^5U +48L^7] + 2L^2U^2L - 8U L^5 -8LU L^4\\
& &-4L^3UL^2 + 8L^4U L+ 40L^7 + U LU^2 + [[-2L^2U 2L+ 4LU L^4 - 4L^4UL - 16L^7]\\
& &-4U L^5 + 4L^3UL^2 + 16L^7] + 2U L^3U + 8U L^5,
\end{eqnarray*}}
that is,
$$
ULU^2 = 2\Big(UL^2UL-UL^3U -LUL^2U-L^2ULU+\qquad\qquad 
$$
\begin{equation}\label{ulu2t}
\qquad \qquad\qquad\qquad\quad 2UL^5-2LUL^4-2L^2UL^3
-3L^4UL-L^5U-16L^7 \Big).
\end{equation}
Next, we can reduce the relation  $0= \delta[y,x,x²,x²x²]x^4$ using the above identities. This
yields 
\begin{equation} \label{ul5}
UL^5 = - LUL^4 + \frac{1}{2} L^2UL^3 + \frac{3}{4}L^4UL +\frac{3}{4} L^5U+ 8
L^7. 
\end{equation} 
Now combining (\ref{ulu2t}) and (\ref{ul5}) we obtain 
$
ULU^2 = 2 UL²UL -2 UL^3U -2 LUL^2U-2 L^2ULU -8 LUL^4-2 L^2UL^3-3L^4UL+ L^5U$.
Thus, we have three identities  
 of $x$-degree 7 and $y$-degree 1 which may be written as multiplication
 identities: 
\begin{center}{\footnotesize
Table iii: Multiplication identities of degree 7
\begin{array}{l|rrrrrrrrrr}\hline
&  UL^2UL  &  UL^3U  &  LUL²U  &  L(LU)^2  &  LUL^4 &  L²UL^3 &
 L^3UL² & L^4UL & L^5U & L^7 \\ \hline
 L^3U² & 0&0&0&0&0&0&-2&-1&-5&-20\\
 UL^5  & 0&0&0&0&-1&1/2&0&3/4&3/4&8\\
 ULU^2 &2&-2&-2&-2&-8&-2&0&-3&1&0\\ \hline
\end{array}
}
\end{center}
In an analogous way, using successively the identities
$$
0= \delta[y,x,x,x(x(x^2x²))]x^4,\quad
0= \delta[y,x²,x²,x^2x²]x^4,\quad 
0= \delta[y,x,x^2,x(x^2x²)]x^4,
$$
multiplying  the second  identity of Table ii  with the operator $U$ from the left and replacing the
occurrences of $UUL$, and finally using $0= \delta[y,x,x^3,x^2x²]x^4$,
we obtain  the following 5 multiplication identities:
\begin{center}{\footnotesize
Table iv: Multiplication identities of degree 8
\begin{array}{l|rrrrrrrrr}\hline
 &  UL^4U  &  LUL^2UL  &  LUL^3U  &  L^2UL²U  &  L^2UL^4  &  L^4UL^2  &
  L^5UL  &  L^6U  &  L^8 \\
\hline
 L^3ULU  & 0 & 0 & 0 & 0 & 0 &  -1/2  &  -2  &  -11/2  &  -20 \\
 UL²U^2  & -4 & -2& -2& 0 & 2& -5/2 & 13 & 31/2 & 32 \\  
 (UL²)^2  & 0 & 1 & 0 & -1 &-12 &-11/4&-7/2&25/4&36\\
 UL^3UL  &-1&-1&-1&0&-4&-11/2&-3&9/2&0\\
 L^3UL^3 &0&0&0&0&0&-3/4&-3/2&-3/4&-8\\ \hline
\end{array}
}
\end{center}
 Now, relations $0= \delta[y,x,x^3,x(x^2x^2)]x^4$, 
$0= \delta[y,x,x^2,x(x(x^2x^2))]x^4$, $0= \delta[y,x^2,x^2,$ $(x^2x^2)]x^4$,
$0= \delta[y,x, x^2x^2, x^2x^2]x^4$, $0= \delta[y,x^2, x^3,x^2x^2]x^4$, and
multiplying the relation determined by the last  row of Table iii with the
operator $U$ from the left and  first replacing the
occurrences of  $UUL$, imply the following 6
multiplication  identities:
\begin{center}{\footnotesize
Table v: Multiplication identities of degree 9\\
\begin{array}{l|rrrr}\hline
 & LUL^4U & (L^2U)^2L & L^7U & L^9\\
\hline
L^6UL     & 0 & 0 &-7 &- 48\\
L(L²U)^2  & 0 & 0 & - 217 &- 4510/3 \\
UL^4UL    & 1 & 0 & -587/2 & -6155/3\\
L^2UL^3U  & 0 & 0 & 29/3 & 422/9\\
UL^2ULU   & 0 & 0 & 1318/3 & 27988/9\\
L^5UL^2   & 0 & 0 & -23 & -496/3\\ \hline
\end{array}
}
\end{center}

The author used a MAPLE language program to discover these identities. 
We now present a solution of a Conjecture of \cite{CHL} in the positive. 
We see that for every $a\in \A$, the associative algebra $\A_a$, generated by 
the  multiplication operators $L_a$ and $L_{a^2}$,  is nilpotent of index $\leq 10$.

\begin{Th}\label{Th2}  Let $\A$ be an algebra over a field $F$ of characteristic
  $\ne 2,3$, satisfying $x^4 = 0$. Then every monomial in $P$ of
  $x$-degree $\geq 10$ and $y$-degree 1 is an identity in $\A$. In particular,
$L_a^{10} = 0$ for  all $a\in \A$. 
\end{Th}
\begin{proof}
First we shall prove that every monomial of $x$-degree 10 and $y$-degree 1 is an identity in $\A$.
 Multiplying the operators in the first line of  Table v with $L$ from the left and from the right,  and  
 the operators in the first line of  Table iv with $U$ from the left and from the right and next using 
 reductions from Tables i-v we see that we only need to prove that  $L^2UL^4U =0, L^8U=0$ and 
 $L^{10}=0$  are multiplication identities in $\A$.  Now, for any $x$ in $\A$ we have
 {\small
 \begin{eqnarray*}[ll]
 L^7UL &= L(L^6UL) =  -7 L^8U- 48 L^{10},\\
 L^6UL^2&=(L^6UL)L = -7L^7UL-48L^{10} = 49 L^8U +288 L^{10},\\
 L^6UL^2&=L(L^5UL^2)= -23 L^8U -496/3L^{10}.
 \end{eqnarray*}}
 Therefore 
\begin{eqnarray}\label{eqL1}
	27 L^8U+ 170 L^{10}=0.	
\end{eqnarray}
Now, 
 {\small
 \begin{eqnarray*}[ll]
 L^5UL^3&=(L^5UL^2)L = -23L^7UL -496/3L^{10} = 161 L^8U +2816/3 L^{10},\\
 L^5UL^3&=L^2(L^3UL^3)= -3/4L^6UL^2 - 3/2 L^7UL-3/4L^8U - 8 L^{10}\\
        &=-27 L^8U -152 L^{10},
 \end{eqnarray*}}
 and hence  
\begin{eqnarray}\label{eqL2}
	141 L^8U + 818 L^{10}=0.
\end{eqnarray}
Next      
  {\small
 \begin{eqnarray*}[ll]      
 L^3UL^3U&=L(L^2UL^3U)= 29/3L^8U +422/9L^{10},\\
 L^3UL^3U&=(L^3UL^3)U=-3/4 L^4UL^2U-3/2 L^5ULU -3/4L^6UU - 8L^8U\\
        &=-3/4 L(L^3UL^2U) -3/2L^2(L^3ULU)-3/4 L^3(L^3U^2)-8L^{10}\\
        &= 9/4L^6UL^2 +15/4L^7UL  +667/4 L^8U  +2345/2L^{10}\\
        &= 1003 L^8U  +3281/2 L^{10},  
 \end{eqnarray*}}
 so that 
\begin{eqnarray}\label{eqL3}
	17880 L^8U + 28685 L^{10}= 0.
\end{eqnarray}
Combining (\ref{eqL1}-\ref{eqL3}) we obtain that $L^8U=0$ and $L^{10}=0$. Now, we have by Table v that 
$0 = (L^2UL^3U)L = L^2(UL^3UL) = -L^2UL^4U - L^3UL^2UL - L^3UL^3U - 4 L^4UL^4- 11/2 L^6UL^2
-3L^7UL +9/2L^8U = -L^2UL^4U - (L^3UL^2U)L - 4L(L^3UL^3)L = -L^2UL^4U$. Therefore, we have $L^2UL^4U=0$.

In an analogous way, we can see that every monomial of $x$-degree 11 and $y$-degree 1 is an identity in $\A$.
This proves the theorem.
\end{proof}

Now we shall investigate two subvarieties of $\mathcal{V}$. We start  in Subsection 2.1 with the class of all nilalgebras in $\mathcal{V}$ of
index $\leq 5$ and next   in Subsection 2.2 we study  the multiplication identities of the variety 
of all the nilalgebras  in $\mathcal{V}$ of index $\leq 6$.

 \subsection{The identity x((xx)(xx))=0} We will now  consider 
the class  of all algebras in $\mathcal{V}$  satisfying  the identity $x(x²x²)=0$.  
 First,  linearization $\delta[y]\{x(x^2)^2\}$ implies 
  \begin{equation}\label{eq45.1}
  L_{x^2x^2} = -4 LUL,
 \end{equation}
 and identity $\delta[y]\{x^2x^3\}=0$ forces 
 \begin{equation}\label{eq45.2}
 UU = -2 ULL + 2 LUL + 4 L^4.
 \end{equation}
 Next, using above identity and $\delta[y,x^2]\{x(x^2)^2\}=0$ we get that
 $0 = 4UUL+4LUU+8LL_{x^3}L = 4(UUL+LUU-2LLUL-4L^5)=
 8(-UL^3+LULL+2L^5 -LULL +LLUL+ 2L^5- LLUL - 2L^5) = 8(-UL^3+2L^5)$. Hence
    $UL^3 = 2 L^5$.
Now idnetity $L_{x(x^2x^2)} = 0$ and relations (\ref{eql5}) and (\ref{eq45.2}) imply 
   $L^2UL = -L^3U - 4 L^5$. Thus, we have the following multiplication identities
   \begin{center}{\footnotesize
{\rm Table vi: Multiplication identities of degree 5}\\
\begin{array}{l|rrrr}\hline
&ULU&LUL^2&L^3U&L^5\\
\hline
UUL  & 0 & 2 & 0 & 0\\
LUU  & 0 &-2 &-2 &-4\\ 
L^2UL& 0 & 0 &-1 &-4\\
UL^3 & 0 & 0 & 0 & 2\\ \hline
\label{eq5.0}\label{eq45.3}
\end{array}
}
\end{center}
   

  From Table ii, we can prove that 
  \begin{equation}\label{eql5.1}
  (UL)^2 = - UL^2U - (LU)^2 + 2 L^3UL + 4 L^4U + 16 L^6,
 \end{equation}
 and $\delta[x^2]\{x^2(x(x(xy)))-2 x(x(x(x(xy))))\} =0$  forces
   \begin{equation}\label{eql5.2}
  (UL)^2 + UL^2U + 2 L^3UL + 4  L^6 = 0.
 \end{equation}
 Combining (\ref{eql5.1}) and (\ref{eql5.2}), we have
 $(LU)^2 = 4  L^6$ and $(UL)^2 = -UL^2U + 2 L^4U + 4 L^6$.
Now, we can check easily the following multiplication identities
  \begin{center}{\footnotesize
{\rm Table vii: Multiplication identities of degree 6}\\
\begin{array}{l|rrr}\hline
&ULLU&L^4U&L^6\\
\hline
UUU    &-2 & 4 & 8\\
UULL   & 0 & 0 & 4\\ 
ULUL   &-1 & 2 & 4\\
LUUL   & 0 & 2 & 0\\
LULU   & 0 & 0 & 4\\ \hline
\end{array}\qquad \qquad
\begin{array}{l|rrr}\hline
&ULLU&L^4U&L^6\\
\hline
LLUU   & 0 &-4 &-4\\
UL^4   & 0 & 0 & 2\\
LUL^3  & 0 & 0 & 2\\
L^2UL^2& 0 & 1 & 0\\
L^3UL  & 0 &-1 &-4\\ \hline
\end{array}
}
\end{center}

 \begin{Th}\label{th45.1} Let $\A$ be an algebra over a field $F$ of characteristic
  $\ne 2$ or $3$, satisfying the identities $x^4 = 0$ and $x(x^2x^2) = 0$. Then
   every monomial in $P$ of
  $x$-degree $\geq 7$ and $y$-degree 1 is an identity in $\A$. In particular,
$L_a^{7} = 0$ for  all $a\in \A$. Furthermore, 
 the algebra generated by $L_x$ and $L_{x^2}$ is spanned, as vector space, by \\
  $L,U,L^2,UL,LU,L^3,UL^2,LUL,L^2U,L^4,ULU,LUL^2,L^3U,L^5,UL^2U,L^4U,L^6$.
\end{Th}
\begin{proof} 
We shall prove that every monomial of $x$-degree $\geq 7$ and $y$-degree 1 is an identity in $\A$.
 Multiplying the operators in the first line of  Table vii with $L$ and $U$ from the left and from the right, and  
 the operators in the first line of  Table vi with $U$ from the left and from the right, and next using 
 reductions from Tables i-vii we see that we only need to prove that  $LUL^2U =0, L^5U=0$ and 
 $L^{7}=0$  are multiplication identities in $\A$.  Now, we have 
 $0 = \delta[y,x^2x^2]\{x(x^2)^2\} = 
4L_{x^2x^2}UL + 4LUL_{x^2x^2}= -16 LULUL - 16 LULUL = -32 LULUL =-32 (LU)^2L 
= -2^7 L^7$, so that $L^7=0$. Also $0 = LULUL = L(UL)^2 = - LUL^2U + 2 L^5U$. Therefore,
$LUL^2U = 2 L^5U.$
Finally, from Table vi we have  that 
$0 = (L^2UL + L^3U+ 4L^5)L^2 = L^2UL^3+ L^3UL^2 = L^3UL^2 = L(L^2UL^2) =
 L^5U$. This proves the theorem.
\end{proof}

\subsection{The identity x(x((xx)(xx)))=0} In this subsection we 
consider the class  of all algebras in $\mathcal{V}$  satisfying  the identity $x(x(x²x²))$ $=0$.  
Because  we use linearization process of identities  and $x(x(x^2x^2))$ has degree 6, we  need 
consider the field $F$ of characteristic not 5 (2 or 3.)

From linearization $\delta[y]\{x(x(x²x²))\}$, we get the  multiplication identity $L_{x(x^2x^2)}+ LL_{x^2x^2}+4L^2UL=0$ and now Lemma \ref{Le3} forces 
\begin{equation}
LUU = - 2 LUL^2 -2 L^3U - 4 L^5.
\end{equation}
The relation $0=\delta[y,x^2]\{x(x(x^2x^2))\}= UL_{x^2x^2}+ 4LL^{x^2x^3}L+ 4 ULUL + 4LUUL+ 8L^2L_{x^3}L + 4L^2UU$ implies 
\begin{equation}
LUL^3 = -\frac{1}{2} \big( L^2UL^2 + L^3UL\big),
\end{equation}
since we can use identities from  Tables i-v. Next, by 
$0=\delta[y,x^3]\{x(x(x^2x^2))\}$ and 
$0=\delta[y,x^2,x^2]\{x(x(x^2x^2))\}$ we get 
\begin{eqnarray}[lcl]
L^4UL &= & -3 L^5U - 16 L^7,\label{eq46.3}\\
 L^2ULU & = &  - L^3UL^2 + 5 L^5U + 28 L^7, \label{eql2ulu}
\end{eqnarray}
and identities $0=\delta[y,x^2,x²,x²]\{x(x(x^2x^2))\}$ and 
$0=\delta[y,x^2,x^3]\{x(x(x^2x^2))\}$ imply
\begin{eqnarray}[lcl]
UL^4U & = & -\frac{1}{2} L^2UL^2U + 24 L^6U + 62 L^8,\\
L^2UL^2U & = & 48 L^6U + 156 L^8.
\end{eqnarray}
Now, identity $0= \delta[y,x²x^2]\{x³x^3\}$ forces 
\begin{equation}\label{eq46.6}
L^6U = - 2 L^8.
\end{equation}

\begin{Th} Let $\A$ be a commutative algebra over a field $F$ of characteristic
  not $ 2,3$ or $5$ , satisfying the identities $x^4 = 0$ and $x(x(x^2x^2)) = 0$. Then every monomial in $P$ of
  $x$-degree $\geq 9$ and $y$-degree 1 is an identity in $\A$. In particular,
$L_a^{9} = 0$ for  all $a\in \A$. 
\end{Th}
\begin{proof} By Tables i-v, we only need to prove that $LUL^4U =0$, $L^2UL^2UL = 0$, $L^7U=0$ and $L^9=0$ are
 multiplication identities in $\A$. From (\ref{eq46.3}-\ref{eq46.6}) may be deduced immediately 
$L^7U =-2L^9$ and  $2L^9 =2L^8 L  = -L^6UL = -L^2(L^4UL) = 3L^7U + 16L^9 = - 6L^9 + 16L^9 = 10 L^9$. Therefore $L^9 = 0$ and $L^7U=0$ are identities in $\A$. Now $L^2UL^2UL= (L^2UL^2U)L= 48L^6UL+ 156L^9 = 0$ and $LUL^4U =L(UL^4U) = -(1/2)L^3UL^2U+ 24L^7U + 62L^9 = -(1/2) L(L^2UL^2U) = -24 L^7U - 78 L^9 = 0$. This proves the theorem.
\end{proof}

\vspace{.5cm}

\noindent \bf Juan C. Gutierrez Fernandez \rm\\
\noindent Departamento de Matemática-IME,\\
\noindent  Universidade de São Paulo\\
\noindent  Caixa Postal 66281,\\
\noindent   CEP 05315-970,\\ São Paulo, SP, \\ Brasil\\
\noindent e-mail: jcgf@ime.usp.br
\end{document}